\documentclass{amsart}
\usepackage[cp1251]{inputenc}   

\usepackage{amsmath,amssymb,mathrsfs,bbm,yhmath,longtable,mathtools}
\usepackage{amsthm,rotating,xcolor,epsfig,hyperref,url,MnSymbol,rotate}
\usepackage{graphicx}
\usepackage[weather]{ifsym}

\usepackage{caption}
\usepackage{subcaption}

\usepackage{tikz}
\usetikzlibrary{knots}
\usetikzlibrary{decorations.pathmorphing,decorations.markings}
\usetikzlibrary{arrows,positioning}
\usetikzlibrary{svg.path}

\usepackage[sort,compress]{cite}

\theoremstyle{plain}
\newtheorem{lemma}{Lemma}
\newtheorem{theorem}{Theorem}
\newtheorem{definition}{Definition}
\newtheorem{corollary}{Corollary}

\newtheorem{remark}{Remark}

\newcommand{\MyCal}[1]{{\mathcal {#1}}}

\newcommand{\sign} {\operatorname{sign}}

\newcommand{\cross}[1]{\#{#1}}

\newcommand{\Input}{{\operatorname{inp}}}
\newcommand{\Output}{{\operatorname{out}}}
\newcommand{\SignRing}{\mathbb{C}}

\newcommand{\Left}[1]{{\operatorname{L}(#1)}}
\newcommand{\Right}[1]{{\operatorname{R}(#1)}}
\newcommand{\re}{\operatorname{Re}}
\newcommand{\im}{\operatorname{Im}}

\title[A {K}auffman bracket polynomial for the non-orientable thickening of a non-orientable surface]
{An analog of the {K}auffman bracket polynomial for knots in the non-orientable thickening of a non-orientable surface}
\author{Vladimir Tarkaev}
\thanks{The work is supported by RSF (grant number 23-21-10014).}
\address{\noindent Chelyabinsk State University, Chelyabinsk, Russia}
\address{\noindent Krasovskii Institute of Mathematics and Mechanics, Ural Branch of the Russian Academy of Sciences, Yekaterinburg, Russia}
\email{trk@csu.ru}

\begin{document}

\begin{abstract}
We study pseudo-classical knots in the non-orientable thickening of a non-orientable surface,
specifically knots that are orientation-preserving paths
in a non-orientable $3$-manifold of the form (non-orientable surface) $\times$ $[0, 1]$.
For these knots, we propose an analog of the Kauffman bracket polynomial.
The construction of this polynomial closely mirrors the classical version,
with key differences in the definitions of the sign of a crossing and the positive/negative smoothing of a crossing.
We prove that this polynomial is an isotopy invariant of pseudo-classical knots
and demonstrate that it is independent of the classical Kauffman bracket polynomial for knots in the thickened orientable surface,
which is the orientable double cover of the non-orientable surface under consideration.


\end{abstract}

\maketitle

\section*{Introduction}
\label{sec:Introduction}  

Initially, the theory of knots studied knots in the $3$-sphere,
but now many other kinds of object are also considered in the theory.
In the context of the present paper, first of all we need to mention
knots in a thickened surface,
that is in an orientable $3$-manifold of the form
(orientable surface) $\times$ $[0,1]$.
Some authors consider the case of knots in a thickened non-orientable surface
(\cite{Bourgoin2008},\cite{Drobotukhina1990},\cite{KamadaKamada2016},\cite{Nabeeva2016})
where the thickening  is understood
as an orientable $3$-manifold equipped with an $I$-bundle structure over a surface
which is not necessarily orientable  (here $I$ denotes a segment).
Note that in all these cases the thickening is assumed to be orientable independently whether the base surface is orientable or not.

Knots in non-orientable $3$-manifolds are also studied 
(see, for example, \cite{Vassiliev1998}, \cite{Vassiliev2016})
but, unfortunately, the subject is not very popular.
Maybe, one of the reasons is that some important constructions cannot be transferred straightforwardly from an orientable to a non-orientable case.
In particular, that is so for the Kauffman bracket polynomial (\cite{Kauffman1987},\cite{Kauffman1999})
and many other polynomial knot invariants
that explicitly use in their construction the notion of the sign of a crossing
(see, for example, \cite{Henrich2010},\cite{Kauffman2013},\cite{Tarkaev2020_1},\cite{Vesnin2018}).
In the classical situation, the sign of a crossing $x$ of an oriented diagram $D$ on a surface $\Sigma$
is defined to be $1$ (resp. $-1$) if a pair
(a positive tangent vector of overgoing branch at $x$, a positive tangent vector of undergoing branch at $x$)
is positive (resp. negative) basis
in the tangent space of $\Sigma$ at $x$.
Clearly, the definition cannot be used if $\Sigma$ is non-orientable.
One more example is the notion of positive and negative smoothing of a crossing (in other terms, the smoothing of the type~$A$ and of the type~$B$)
which play a key role in the construction of the Kauffman bracket polynomial.
The definition of the latter notion uses concepts of ``right'' and ``left'' which have no sense on a non-orientable surface.
Note that aforementioned notions can be reformulated
so that they have sense in the case of the orientable thickening of a non-orientable surface.
However, the approach allowing to do this cannot be used
in the case of the non-orientable thickening of a non-orientable surface.

In the present paper, we study knots in the non-orientable thickening of a non-orientable surface,
that is  a non-orientable $3$-manifold of the form $\Sigma \times [0,1]$,
where $\Sigma$ is a non-orientable surface with (maybe) non-empty boundary.
We propose an approach
that in some particular case gives correct definitions of the sign of a crossing
and of two types of a smoothing.
The knots we mean (we call them pseudo-classical) are those
which are orientation-preserving paths in the manifold.
Two aforementioned definitions allow defining an analog of the Kauffman bracket polynomial for pseudo-classical knots.
The construction we use word-for-word coincides with the classical one,
the specificity of the non-orientable thickening is hidden in using definition of positive and negative smoothing.
Our proof of the invariance of the polynomial (we denote it by $J$) is close to the classical one but do not coincide with the latter.
Finally, we compare the polynomial $J$ with its natural competitor --- with the Kauffman bracket polynomial for knots in the thickened orientable surface, 
which is the orientable double cover of the non-orientable surface under consideration.
We demonstrate that none of  these invariants is a consequence of the other.
Namely, we consider two pairs of knots in the non-orientable thickening of the Klein bottle
so that knots forming the first pair have coinciding polynomials $J$
while $2$-component links in the thickened torus those are their double covers have distinct Kauffman bracket polynomials.
Knots forming the second pair, on the contrary, have distinct polynomials $J$ while their double covers have coinciding Kauffman bracket polynomials.

The paper is organized as follows.
In Section~\ref{sec:Definitions},
we give some definitions, including a new definition of the sign of a crossing.
In Section~\ref{sec:BracketPolynomial},
we define two types of smoothing.
Then, we define the polynomial $J$ and prove its invariance.
In the end of this section, we briefly discuss a generalization of the polynomial
analogous to the homological version of the Kauffman bracket polynomial for knots in a thickened orientable surface
~\cite{KauffmanDye2005}.
In Section~\ref{sec:vs},
we compare the polynomial $J$ and the Kauffman bracket polynomial for orientable double cover.
In particular, in Section~\ref{sec:J is stronger},
we explicitly compute the polynomial $J$ for two knots in the non-orientable thickening of the Klein bottle.
In Appendix, we give source data
needed for computing the values  of the Kauffman bracket polynomial
in the format acceptable by the computer program ``3--manifold recognizer''.

\section{Definitions}
\label{sec:Definitions}

\subsection{Knots and diagrams}
\label{sec:Knots}

Throughout  this paper $\Sigma$ denotes a non-orientable surface with (maybe) non-empty boundary
and $\hat{\Sigma}$ denotes the non-orientable manifold of the form $\Sigma \times [0,1]$
with fixed structure of the direct product.
The latter condition becomes crucial for us in the case
of surface with non-empty boundary  
when there are aforementioned structures of the corresponding manifold
with homeomorphic but non-isotopic base surfaces.

Knots in $\hat{\Sigma}$ and their diagrams on $\Sigma$ can be defined by analogy with the case
of knots in a thickened orientable surface
(below, we will refer to the latter case  as the ``classical case'').
Knots in $\hat{\Sigma}$ will be considered up to ambient isotopy.

In our case, two diagrams on $\Sigma$ represent the same knot
if and only if they are connected by a finite sequence of Reidemeister moves $R_1,R_2,R_3$
(which are just the same as classical ones)
and ambient isotopies.
indeed, let $\Sigma$ be represented by polygon
of which sides are assumed to be identified (maybe) with a twist
and a diagram is drawn inside the polygon.
If we isotope the knot in question inside the direct product of the polygon and the segment,
then we can refer to the corresponding classical fact.
Otherwise, it is necessary to consider a few auxiliary moves.
These moves describe how we push arcs and crossings  of the diagrams
through the boundary of the polygon.
In the case of knots in orientable thickening of non-orientable surface
(see, for example, \cite{Drobotukhina1990}, \cite{Bourgoin2008}),
some of these moves permute overgoing and undergoing branches at pushing crossing.
In our case, $\hat{\Sigma}$ is the direct product
hence we have no moves permuting overgoing and undergoing branches,
i.e., all auxiliary moves represent ambient isotopies of a diagram.
Therefore, in proofs below, we can restrict ourselves to Reidemeister moves $R_1,R_2,R_3$.

\subsection{Pseudo-classical knots}
\label{sec:AlmostClassical}

A knot $K \subset \hat{\Sigma}$ is called an \emph{pseudo-classical}
if it is an orientation-preserving path in $\hat{\Sigma}$.
In other words, a knot in $\hat{\Sigma}$ is pseudo-classical
if its regular neighborhood in the manifold is homeomorphic to the solid torus.
Clearly, in a non-orientable manifolds not all knots have the property
since, in this case, there are knots whose regular neighborhood is homeomorphic to the solid Klein bottle.
However, the latter kind of knots is out of consideration in the present paper.

Let $K$ be a pseudo-classical oriented knot represented by a diagram $D \subset \Sigma$.
Since the knot is pseudo-classical,
the diagram (or, more precisely, the projection of $K$ onto $\Sigma$)
viewed as a closed directed path on the surface
is an orientation-preserving  path.
Hence, $2$-cabling $D^2 \subset \Sigma$ of $D$
is a diagram of a $2$-component link
(recall that $2$-cabling of a diagram is, informally speaking, the same diagram
drawn by doubled-line with preserving over/under-information in all crossings).
Note again that this is not the case for an arbitrary diagram on $\Sigma$.
This is because $\Sigma$ is assumed to be non-orientable
hence there are knot diagrams on the surface,
those are orientation-reversing paths.
Hence, their $2$-cabling represents not a $2$-component link but a knot
which goes along the given knot twice.
In the latter case, $2$-cabling bounds the {M}\"obius strip (going across itself near each crossing of the given diagram)
while $2$-cabling of a pseudo-classical knot bounds an annulus (having self-intersections).

In the case of an oriented surface, 
one can speak of left and right components
of $2$-cabling of an oriented knot diagram.
Since $\Sigma$ is non-orientable, we do not have the concepts of left and right.
 However, we will use the terms ``left component'' and ``right component'' of $D^2$
keeping in mind that in our situation this is nothing more than a convenient notation.
A choice (which component is chosen as the left and which is chosen as the right) will be called the \emph{labeling} of $D^2$.
The components will be denoted by $\Left{D}$ and $\Right{D}$, respectively.
Below, $\Left{D}$ and $\Right{D}$ are assumed to be oriented
so that their orientations are agreed with the orientation of $D$.
Sometimes we will rename $\Right{D}$ into $\Left{D}$ and vice versa,
the transformation will be called the \emph{relabeling} of $D^2$.

A single crossing of $D$ becomes in $D^2$ a pattern of $4$ crossings
(see Fig.~\ref{fig:1}).
There are two distinguished crossings  in each of these patterns:
the first one in which components of $D^2$
come into the pattern and the second one in which they go out.
These crossings are called the \emph{input} (resp. \emph{output})  crossing for the crossing $x$ 
and they are denoted by $\Input(x)$ (resp. $\Output(x)$\,),
where  $x \in \cross{D}$ is that crossing of $D$ to which the pattern corresponds.
Here and below we denote by $\cross{D}$ the set of crossings of a diagram $D$.

Note if the surface under consideration is orientable
then input and output crossings are necessarily intersections of distinct components of $2$-cabling
while in non-orientable case
the situation when input and output crossings are a self-crossing of the same component of $D^2$ may occur.

\begin{figure}
\centering 
\begin{subfigure}[b]{0.45\textwidth}
\centering
\psfig{figure=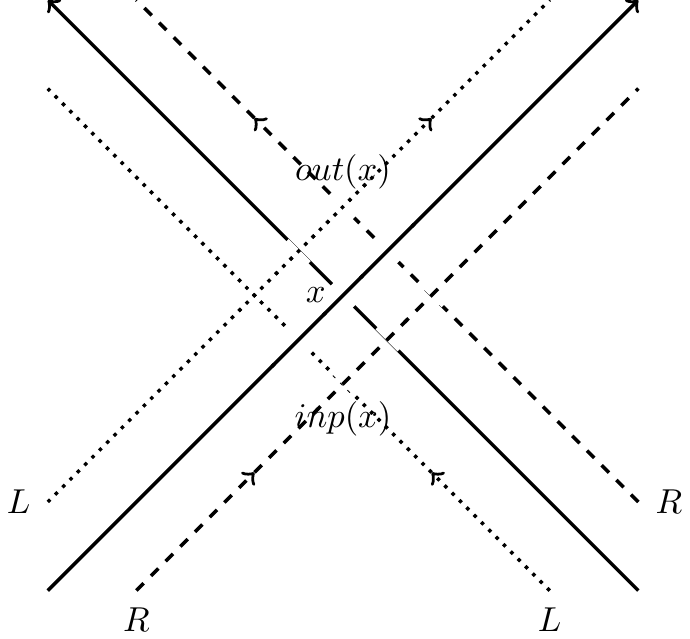, width=0.8\textwidth}
\caption{}
\label{fig:1_1}
\end{subfigure}
\begin{subfigure}[b]{0.45\textwidth}
\centering
\psfig{figure=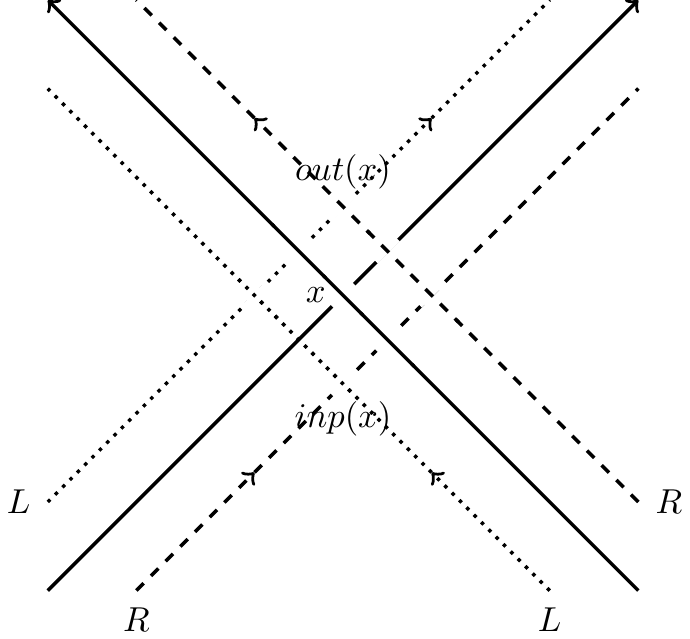, width=0.8\textwidth}
\caption{}
\label{fig:1_2}
\end{subfigure}
\begin{subfigure}[b]{0.45\textwidth}
\centering
\psfig{figure=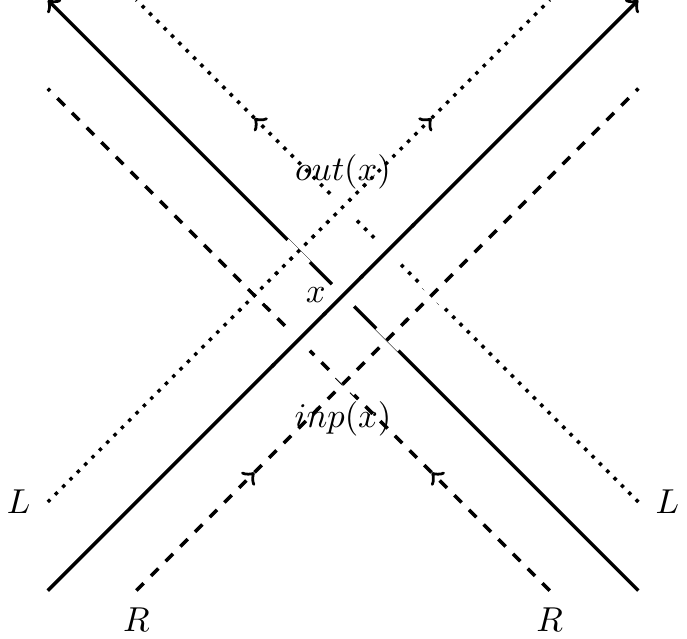, width=0.8\textwidth}
\caption{}
\label{fig:1_3}
\end{subfigure}
\begin{subfigure}[b]{0.45\textwidth}
\centering
\psfig{figure=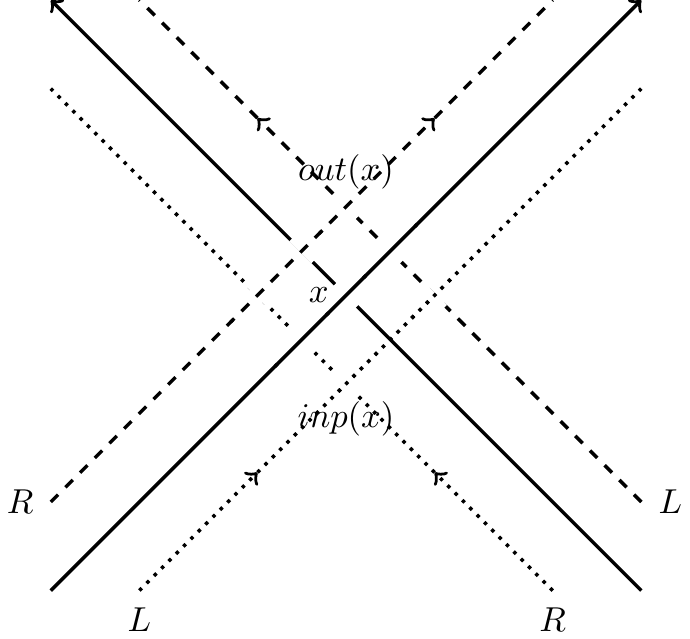, width=0.8\textwidth}
\caption{}
\label{fig:1_4}
\end{subfigure}
\caption{Segments belong to $D, \Right{D}$ and $\Left{D}$ are drawn by solid, dashed and dotted lines, respectively.
}
\label{fig:1}
\end{figure}

\subsection{The sign of a crossing}
\label{sec:Sign}

Recall that in the classical situation the sign of a crossing $x$ of an oriented diagram $D$ on an oriented surface $F$
is defined to be $1$ (resp. $-1$) if the pair $(t_1, t_2)$
is a positive (resp. negative) basis
in the tangent space  of $F$ at $x$,
where $t_1$ and $t_2$ are, respectively, a positive tangent vector of overgoing and undergoing branches of $D$ at $x$.
Clearly, the  definition cannot be used if the surface  is non-orientable.
In the following definition, we use a labeling of  $2$-cabling of the diagram  under consideration instead of the orientation of the surface
in which the diagram lies.

\begin{definition} \label{def:SignForKnot}
Let $D \subset \Sigma$ and $D^2$ be a diagram of an oriented pseudo-classical knot and its $2$-cabling.
The \emph{sign} of a crossing $x \in \cross{D}$ is defined to be the mapping
$\sign : \cross{D} \to \SignRing$
given by the following rule:
\begin{equation} \label{eq:Sign}
\sign(x) =\begin{cases}
1 & \text{if $\Right{D}$ goes over $\Left{D}$ in $\Input(x)$ (Fig.~\ref{fig:1_1}),} \\
-1 & \text{if $\Left{D}$ goes over $\Right{D}$ in $\Input(x)$ (Fig.~\ref{fig:1_2}),} \\
i & \text{if $\Right{D}$ goes over itself in $\Input(x)$ (Fig.~\ref{fig:1_3}),} \\
-i & \text{if $\Left{D}$ goes over itself in $\Input(x)$ (Fig.~\ref{fig:1_4}),}
\end{cases}
\end{equation}
where $i \in \SignRing$ is the imaginary unit. 
\end{definition}

We will say that a crossing $x \in \cross{D}$ is of the \emph{type $1$}
if $\Input(x)$ is an intersection of two distinct components of $D^2$,
otherwise we will say that $x$ is of the \emph{type $2$}.
In other words, a crossing $x$ is of the type~$1$ if $\sign(x) =\pm 1$
and $x$ is of the type~$2$ if $\sign(x) =\pm i$.

\begin{remark} \label{remark:PathsAndTypeOfCrossing}
A crossing $x \in \cross{D}$ decomposes $D$ into two closed paths $p_1(x),p_2(x)$.
Both paths start at $x$ in the positive direction (the first one by the overgoing arc, the other one by the undergoing arc), then the paths go along $D$ in positive direction
until the first return to $x$.
throughout  this paper we consider pseudo-classical knots only,
of which diagrams, as we mentioned above, viewed as a directed path on $\Sigma$
is an orientation-preserving path.
Hence, if $x$ is of the type~$1$
then both of $p_1(x)$ and $p_2(x)$ are paths preserving the orientation,
while if $x$ is of the type~$2$ both of $p_1(x)$ and $p_2(x)$ reverses the orientation.
Therefore, the situation when one of $p_1(x),p_2(x)$ preserves the orientation while the other does not
is impossible in our context.
\end{remark}

Recall that the \emph{crossing change operation} in a crossing $x \in \cross{D}$ is the transformation of $D$
so that overgoing arc in $x$ becomes undergoing and vice versa
while the rest of the diagram remains unchanged.

The following lemma is a direct consequence  of our definition of the sign
and the observation that the input crossing of each $4$-crossing pattern becomes the output crossing and vice versa
as a result of the reversing orientation of $D$.

\begin{lemma} \label{lemma:KnotSignProperties}
$\,\,$\\
{\bf (1)}.
If a crossing $x \in \cross{D}$ is of the type~$1$,
then each of the following transformations of $D$ and $D^2$ multiplies the sign of $x$ by $-1$:
\begin{itemize}
\item[a)] a crossing change operation, 
\item[b)] the reversing of the orientation of $D$,
\item[c)] the relabeling of $D^2$.
\end{itemize}

{\bf (2)}.
If a crossing $x \in \cross{D}$ is of the type~$2$,
then $\sign(x)$ is preserved under the crossing change operation,
while the relabeling of $D^2$ and the reversing orientation of $D$ multiply the sign of $x$ by $-1$.
\end{lemma}

\begin{remark}
In the classical situation, when both the surface and its thickening are orientable
it is natural to choose as $\Right{D}$ that component of $D^2$ that goes to the right of $D$ with respect to its actual orientation.
In the orientable case, all crossings are of the type $1$
and the rule~\eqref{eq:Sign}  gives the standard sign of a crossing.
In this case, the reversing orientation of the knot implies
simultaneous reversing orientation of $D$ and the relabeling of $D^2$,
hence, by Lemma~\ref{lemma:KnotSignProperties}(1),
the signs of a crossing does not depend on the orientation of the diagram.
\end{remark}

Using the  aforementioned definition of the sign, we can define an analog of the writhe number of a diagram
\[
w(D) =\sum_{x \in \cross{D}}\sign(x) \in \SignRing
\]
and two more writhe numbers:
\begin{eqnarray}
w_1(D) =\sum_{\substack{x \in \cross{D}, \\ \sign(x) =\pm 1}}\sign(x)
=\re w(D) \in \mathbb{R},
\label{eq:w_1} \\
w_2(D) =\sum_{\substack{x \in \cross{D}, \\ \sign(x) =\pm i}}\im \sign(x)
=\im w(D) \in \mathbb{R}.
\label{eq:w_2}
\end{eqnarray}

As in the classical case, the first Reidemeister move change the value $w(D)$,
however, the following holds:

\begin{lemma} \label{lemma:w_2}
If diagrams $D_1,D_2$ represent the same oriented pseudo-classical knot $K \subset \hat{\Sigma}$ then
\begin{equation} \label{eq:W_2Invariant}
|w_2(D_1)| =|w_2(D_2)|.
\end{equation}
\end{lemma}
Therefore, the absolute value of the writhe number $w_2$ is an isotopy invariant of a pseudo-classical knot.

\begin{proof}
We need to consider three Reidemeister moves.

The move $R_1$ creates/removes a crossing of the type~$1$ (Fig.~\ref{fig:2_1}),
so it does not affect the value $w_2$.

Both vertices involved in the move $R_2$ (Fig.~\ref{fig:2_2}) are of the same type (either type~$1$ or type~$2$),
and the reversing the orientation of $D$ and the relabeling  of $D^2$
preserve the type.
If these vertices are of the type~$1$ then, as in the classical case, their signs are equal to $+1$ and $-1$.
If the vertices are of the type ~$2$  their signs are equal to $+i$ and $-i$.
Hence, in both cases the sum of the signs is equal to $0$
and values of all three writhe numbers ($w,w_1,w_2$) are preserved under the move $R_2$.

The move $R_3$ (Figs.~\ref{fig:2_3},\ref{fig:2_4}) also does not affect all three writhe numbers.
Therefore, all Reidemeister moves do not change $w_2$.
However, we are forced to use the absolute value in~\eqref{eq:W_2Invariant}
because there is no a canonical way to choose which component of $D^2$ is the right one.
Hence,  (see Lemma~\ref{lemma:KnotSignProperties}(2))
the value $w_2(D)$is defined up to
multiplication by $-1$.

\begin{figure}[h!]
\centering 
\begin{subfigure}[b]{0.45\textwidth}
\centering
\psfig{figure=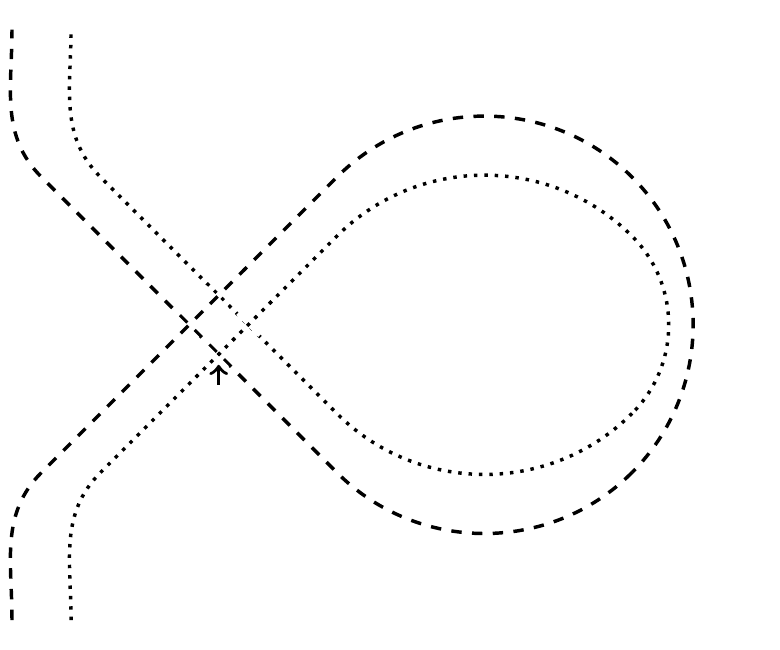, width=0.8\textwidth}
\caption{The move $R_1$}
\label{fig:2_1}
\end{subfigure}
\begin{subfigure}[b]{0.45\textwidth}
\centering
\psfig{figure=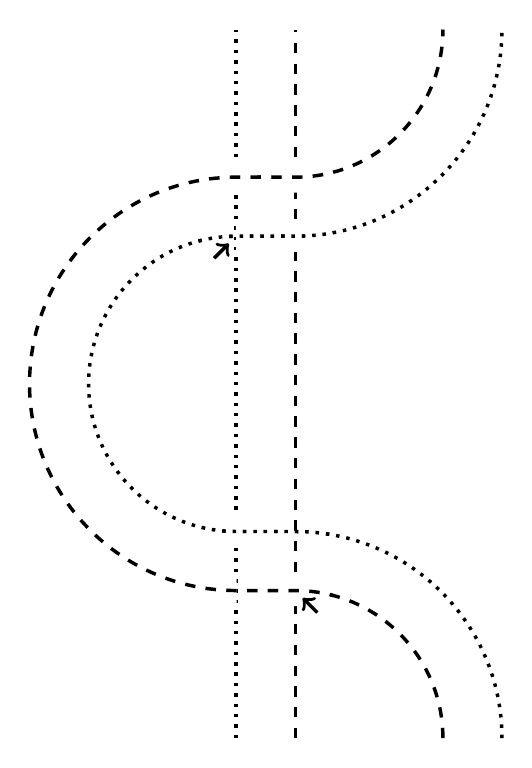, width=0.8\textwidth}
\caption{The move $R_2$}
\label{fig:2_2}
\end{subfigure}
\begin{subfigure}[b]{0.45\textwidth}
\centering
\psfig{figure=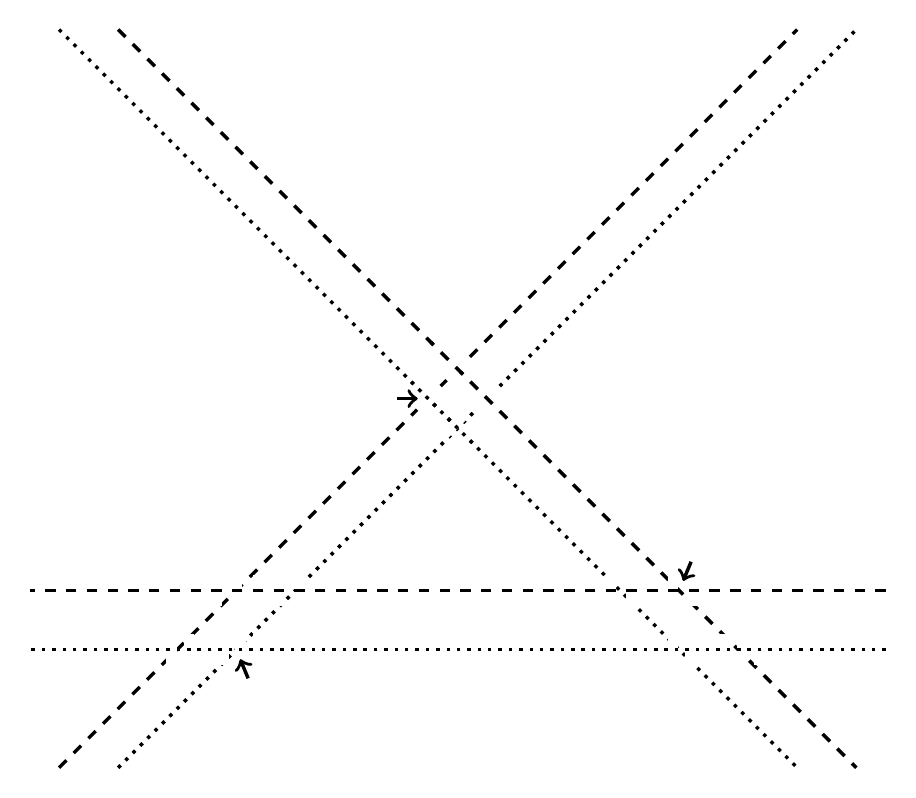, width=0.8\textwidth}
\caption{Before the move $R_3$}
\label{fig:2_3}
\end{subfigure}
\begin{subfigure}[b]{0.45\textwidth}
\centering
\psfig{figure=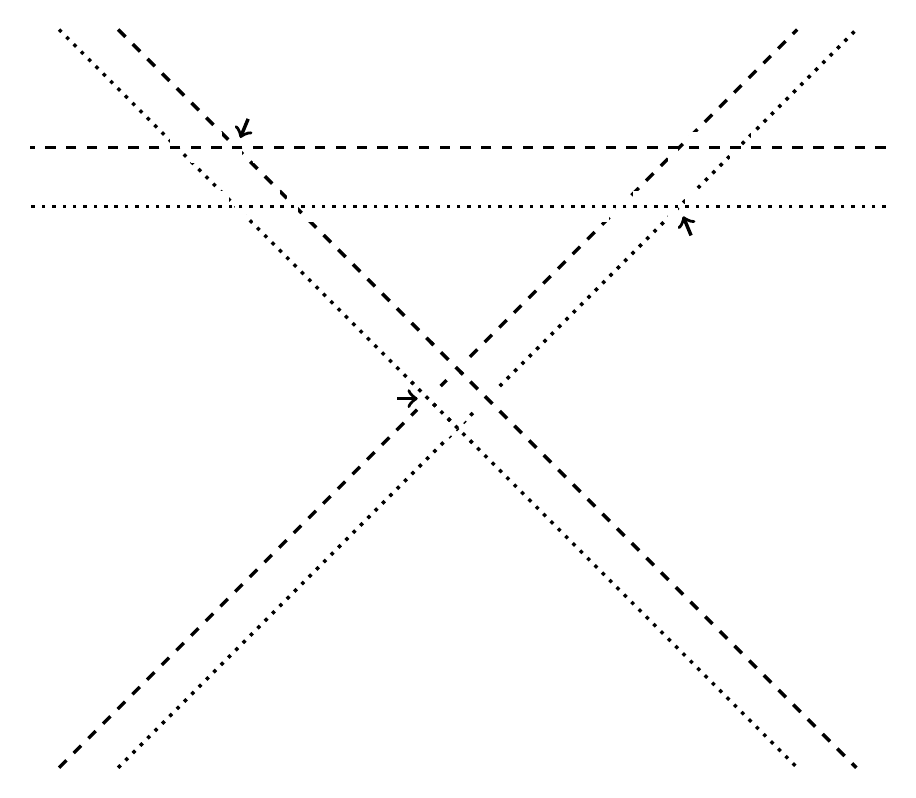, width=0.8\textwidth}
\caption{After the move $R_3$}
\label{fig:2_4}
\end{subfigure}
\caption{\small{
$\Right{D}$ and $\Left{D}$ are drawn by dashed and dotted lines, respectively.
Instead of explicit specifying of the orientation $\Right{D}$ and $\Left{D}$ we draw short arrows which point to input crossings of $4$-crossings patterns
}}
\label{fig:2}
\end{figure}

\end{proof}

\section{An analog of the Kauffman bracket polynomial}
\label{sec:BracketPolynomial}

\subsection{Two types of smoothing of a crossing}
\label{sec:Smoothing}

To use Kauffman's construction of the bracket polynomial \cite{Kauffman1987}
we need to distinguish 
two types of smoothing of a crossing.
We will name them as positive and negative.
Recall that, in the classical case, a smoothing of a crossing $x$ in a diagram $D$ is said to be positive
if it consists in pairing of the following regions adjacent to $x$:
the region lying on the right when we go to $x$ along the overgoing branch of $D$ (the direction in which we go does not matter)
with the region which is opposite to the first one.
The other smoothing is said to be negative.
But, equivalently, one can define positive and negative smoothing via the sign of a crossing
and the notions of smoothing along/across orientation of the diagram.
Below, we use the latter approach.

We begin with some auxiliary terminology.
Let $D \subset \Sigma$ be a diagram representing a pseudo-classical knot.
Consider a crossing $x \in \cross{D}$.
There is a natural correspondence between angles adjacent to $x$ and crossings forming $4$-crossing pattern corresponding to $x$ in $D^2$.
A smoothing of a crossing is a pairing of two vertical angles
or, equivalently, a pairing of two diagonal crossings in the pattern.
Hence, to specify a smoothing, it is sufficient to specify a pair of diagonal crossings (or even one of them only) in the pattern.
We will say that a smoothing of $x$ is \emph{along the orientation} of $D$
if the smoothing is determined by $\Input(x)$
(see Fig.~\ref{fig:3_1}).
The other smoothing is said to be \emph{across the orientation} of the diagram (see Fig.~\ref{fig:3_2}).
Note that $\Input(x)$ and $\Output(x)$ determined the same smoothing.
Below, in figures, we will specify that type of smoothing which we want to use
via a short segment (which can be viewed as a diagonal of the pattern)
having endpoints in those vertical angles
which are pairing in the smoothing.
Therefore, Figs.~\ref{fig:3_3} and~\ref{fig:3_4} represent smoothing of the same type as in Figs.~\ref{fig:3_1} and~\ref{fig:3_2}, respectively.

\begin{figure}[h!]
\centering 
\begin{subfigure}[b]{0.45\textwidth}
\centering
\psfig{figure=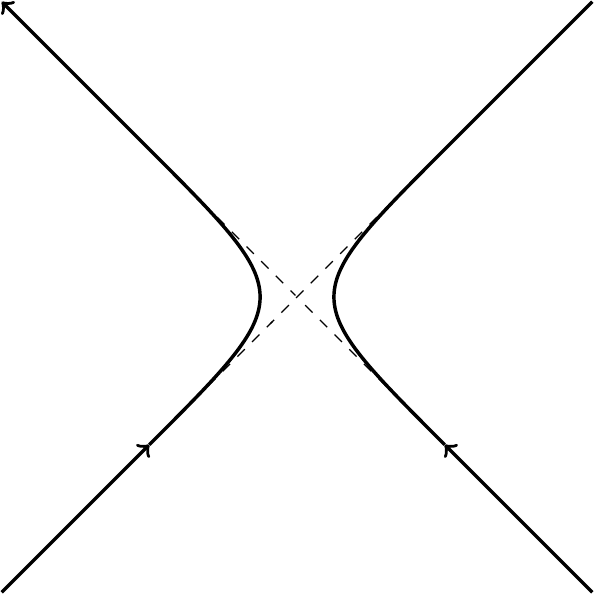, width=0.4\textwidth}
\caption{}
\label{fig:3_1}
\end{subfigure}
\begin{subfigure}[b]{0.45\textwidth}
\centering
\psfig{figure=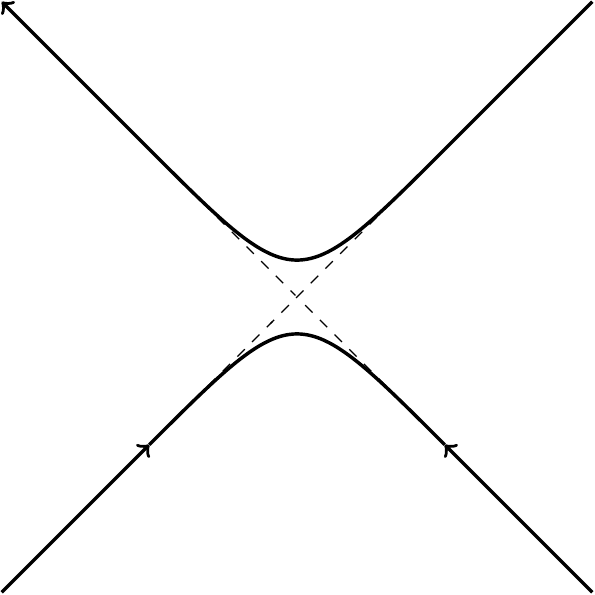, width=0.4\textwidth}
\caption{}
\label{fig:3_2}
\end{subfigure}
\begin{subfigure}[b]{0.45\textwidth}
\centering
\psfig{figure=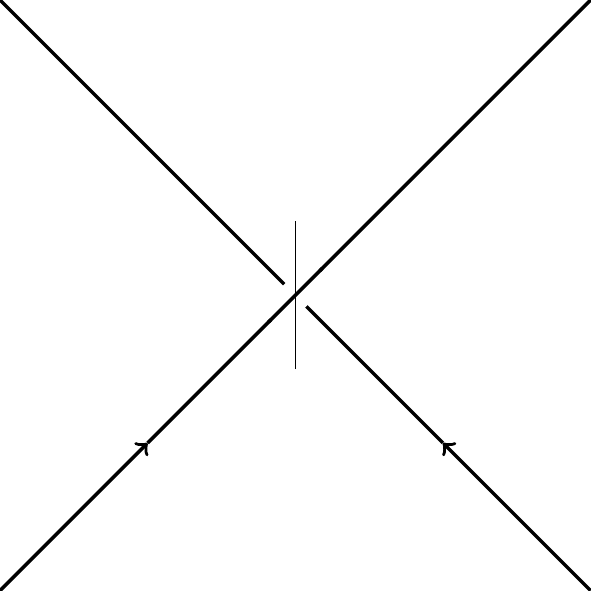, width=0.4\textwidth}
\caption{}
\label{fig:3_3}
\end{subfigure}
\begin{subfigure}[b]{0.45\textwidth}
\centering
\psfig{figure=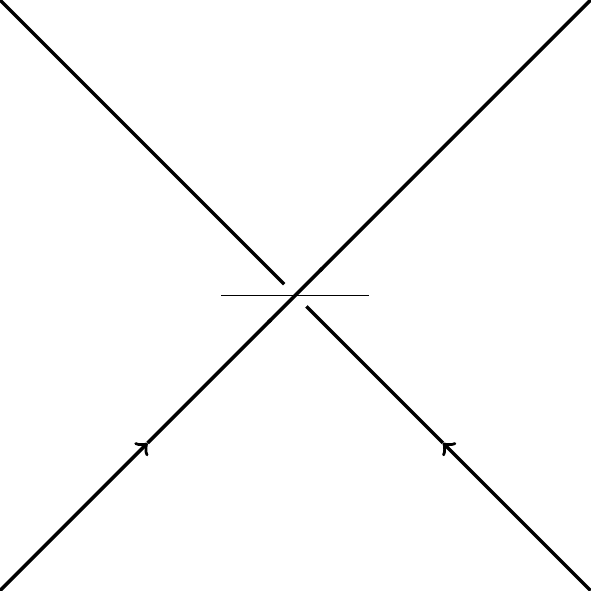, width=0.4\textwidth}
\caption{}
\label{fig:3_4}
\end{subfigure}
\caption{Smoothing along (A) and across (B) the orientation of the diagram
and their simplified drawing (C) and (D), respectively.}
\label{fig:3}
\end{figure}

\begin{definition} \label{def:Smoothing}
Let $D \subset \Sigma$ be an oriented  diagram of an oriented pseudo-classical knot.
A smoothing of  a crossing $x \in \cross{D}$ is called \emph{positive}
if it is determined by the following rule:
\[
\begin{split}
\text{along the orientation } & \text{ if either $\sign(x) =1$ or $\sign(x) =i$,} \\
\text{across the orientation } & \text{ if either $\sign(x) =-1$ or $\sign(x) =-i$.} 
\end{split}
\]
The other  smoothing is called \emph{negative}.
\end{definition}

\begin{lemma} \label{lemma:PropertiesOfSmoothing}
Let $D \subset \Sigma$ be an oriented  diagram of a pseudo-classical knot.
Then,  for a crossing $x \in \cross{D}$
the following holds:

{\bf (1)}.
The reversing orientation of $D$ and the relabeling of $D^2$
make positive smoothing into negative and vice versa,
independently of whether $x$ is of the type~$1$ or of the type~$2$.

{\bf (2)}.
The crossing change operation makes positive smoothing into negative and vice versa
if $x$ is of the type~$1$ and it does not affect the type of smoothing otherwise.

{\bf (3)}.
The positive smoothing in $x$ is determined by that crossing in the corresponding pattern
in which overgoing arc of $\Right{D}$ comes into the pattern,
i.e., actual direction of the undergoing  arc of $D$ and actual labeling of undergoing arcs of $D^2$ in the crossing
do not matter which smoothing in $x$ is positive and which one is negative.
\end{lemma}

\begin{proof}
Pick a crossing $x \in \cross{D}$.

{\bf (1)}.
By Lemma~\ref{lemma:KnotSignProperties},
both transformations (the reversing orientation of $D$ and the relabeling of $D^2$)
multiply $\sign(x)$ by $-1$
independently on whether $x$ is of the type~$1$ or of the type~$2$.
Hence, by Definition~\ref{def:Smoothing},
these transformations make positive smoothing into negative and vice versa.

{\bf (2)}.
This fact is again a direct consequence
of Lemma~\ref{lemma:KnotSignProperties},
which states that the crossing change operation multiplies $\sign(x)$ by $-1$ if $x$ is of the type~$1$
and preserves $\sign(x)$ if $x$ is of the type~$2$.

{\bf (3)}.
Denote by $y$ that crossing
in which overgoing arc of $\Right{D}$ comes into the pattern corresponding to $x$.
There are two possibilities:
either $y$ coincides with $\Input(x)$ or not.
In the first case, by Definition~\ref{def:SignForKnot}, $\sign(x) =1$
and, by Definition~\ref{def:Smoothing}, the positive smoothing is the one along orientation,
i.e., it is determined by $\Input(x) =y$.
In the second case, $\sign(x) =-1$ and the positive smoothing is the one across orientation,
i.e., it is again determined by $y$.
\end{proof}

\begin{remark} \label{remark:RedefinitionOfPositiveSmoothing}
Using the third statement of the above lemma 
we can reformulate our definition of positive smoothing to be like the classical one:
a smoothing in a crossing is positive if it is determined by that angle adjacent to the crossing
which lies to the ``right'' when we go to the crossing along overgoing arc in positive direction.
Here there are two crucial differences with the classical definition.
Firstly, we have to use an ersatz-notion of the ``right'' constructing via a labeling of $D^2$
instead of the natural concept  of right coming from the orientation of underlying surface.
Secondly, in the classical situation the direction along which we go to the crossing does not matter
while in the above definition we have to go along the positive direction if we want to use the ``right'' side only. 
\end{remark}

In the classical situation when both the surface and its thickening are orientable
all crossings are of the type~1 and, consequently, their signs equal $\pm 1$,
and our definition of positive and negative smoothing (Definition~\ref{def:Smoothing}) is equivalent to its classical prototype.
In particular, in the classical case, types of smoothing (in the sense of Definition~\ref{def:Smoothing}) are independent on the orientation of a diagram
because in the orientable situation the reversing the orientation of $D$ implies the relabeling of $D^2$
and, by Lemma~\ref{lemma:PropertiesOfSmoothing}(1), simultaneous performing of these two transformations preserves positive and negative smoothing.

\subsection{The bracket polynomial}
\label{sec:Polynomial}

Below, we define an analog of Kauffman bracket polynomial.
The construction is similar to its classical prototype
(\cite{Kauffman1987}).
The specificity of non-orientable thickening is hidden in using definitions of the sign of a crossing
(Definition~\ref{def:SignForKnot})
and of positive and negative smoothing (Definition~\ref{def:Smoothing}).
The only difference is that we add a factor counting  the signs of crossings of the type~$2$.

Let $R$ be a ring $\mathbb{Z}[u^{\pm 1},v]$ of Laurent polynomials with integer coefficients.

Fix an oriented diagram $D$ of an oriented pseudo-classical knot $K \subset \hat{\Sigma}$.

A \emph{state} of the diagram $D$ is defined to be a mapping S: $\cross{D} \to \{-1,1\}$.
As usual, each state of $D$ gives rise to a collection of pairwise disjoint circles
$\gamma_1(S),\ldots,\gamma_{n(S)}(S) \subset \Sigma$
those appear as a result  of smoothing of all crossings of $D$, the type of smoothing performing in $x \in \cross{D}$ is determined by the value $S(x)$:
the smoothing is positive if $S(x) =1$ and negative otherwise.

Let
\begin{equation} \label{eq:P}
P(S;u) = (-u^2 -u^{-2})^{n(S)-1}\prod_{x \in \cross{D}}u^{S(x)} \in \mathbb{R}
\end{equation}
and
\begin{equation} \label{eq:J}
J(D;u,v) =(-u)^{-3w_1(D)}v^{|w_2(D)|} \sum_{S \in \MyCal{S}(D)}P(S;u) \in \mathbb{R},
\end{equation}
where by $\MyCal{S}(D)$ we denote the set of all states of $D$,
(for definitions of writhe numbers $w_1(D),w_2(D)$ see~\eqref{eq:w_1} and~\eqref{eq:w_2}).

\begin{remark} \label{remark:Relabeling}
The value $J(D;u,v)$ depends on orientation of $D$
and actual labeling of $D^2$.
From Lemma~\ref{lemma:PropertiesOfSmoothing}(1)
and the above definition follows
that the reversing of the orientation of $D$
and the relabeling of $D^2$
imply the substitution $u =u^{-1}$
in the polynomial $J(D; u,v)$.
\end{remark}

\begin{theorem} \label{theorem:J}
If diagrams $D_1,D_2 \subset \Sigma$ represent the same pseudo-classical knot,
then either $J(D_1,;u,v) =J(D_2;u,v)$
or $j(D_1;u,v) =J(D_2;u^{-1},v)$.
\end{theorem}
Therefore, the polynomial $J(K;u,v)$ is an isotopy invariant of a pseudo-classical knot $K \subset \hat{\Sigma}$
up to substitution $u=u'^{-1}$.

\begin{remark} \label{remark:DrawingConvention}
Below, to make figures simpler, we use the following conventions.

1. If $\Right{D}$ goes to the right of an arc of oriented diagram $D$ 
then the arc is drawn by solid line,
otherwise when $\Right{D}$ goes to the left the arc is drawn by dashed line.
Here we mean the standard notion of ``to the right'' coming from the standard orientation of the plane of drawing.

2. Instead of explicit specifying the orientation of a diagram,
we draw short arrows which points to input crossings.

These conventions together with standard over/under-information
allow seeing both the type ($1$ or $2$)
and the sign of depicted crossings
(see Fig.~\ref{fig:4}).
\end{remark}

\begin{figure}[h!]
\centering 
\begin{subfigure}[b]{0.45\textwidth}
\centering
\psfig{figure=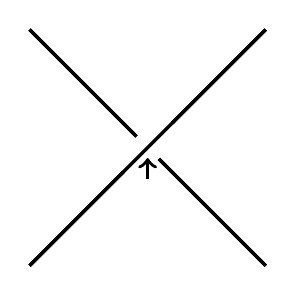, width=0.38\textwidth}
\caption{}
\label{fig:4_1}
\end{subfigure}
\begin{subfigure}[b]{0.45\textwidth}
\centering
\psfig{figure=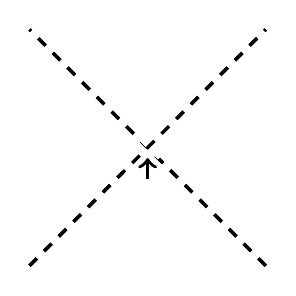, width=0.38\textwidth}
\caption{}
\label{fig:4_2}
\end{subfigure}
\begin{subfigure}[b]{0.45\textwidth}
\centering
\psfig{figure=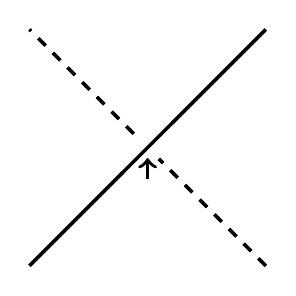, width=0.38\textwidth}
\caption{}
\label{fig:4_3}
\end{subfigure}
\begin{subfigure}[b]{0.45\textwidth}
\centering
\psfig{figure=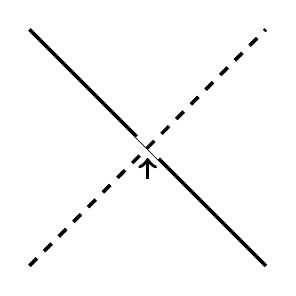, width=0.38\textwidth}
\caption{}
\label{fig:4_4}
\end{subfigure}
\caption{\small{
Depicted crossings are of the type~$1$ in (a), (b) and they are of the type~$2$ in (c),(d).
Their signs are equal to $1,-1,i,-i$, respectively.
}}
\label{fig:4}
\end{figure}

\begin{proof}
The proof below is very close to the one in the classical case.
The invariance of our polynomial $J$ under Reidemeister moves $R_1 -- R_3$ will be established by the same computations.
We only need to explain why we can perform them in our case.

Fix an oriented diagram $D \subset \Sigma$ of an oriented pseudo-classical knot
and a labeling of $D^2$.

{\bf The move $\mathbf{R_1}$}.
By definition of the move, a loop which the move adds/removes bounds a disk,
hence (see Remark~\ref{remark:PathsAndTypeOfCrossing})
a crossing that the move adds/removes is of the type~$1$
and its sign is equal to $\pm 1$.
There are $4$ situations depending on the value of the sign of the crossing
and what passage (over or under) through the crossing we meet first going along modifying fragment in the positive direction.
In all cases,  the invariance of the polynomial $J(u,v)$ under the move $R_1$ can be established 
using the same arguments as in the classical situation.
However, it is necessary to remember that what of smoothing are positive is determined by actual  labeling of $D^2$.
For example, the tangle shown in Fig.~\ref{fig:11_1}
is similar to that in the classical situation:
the sign of the crossing is equal to $+1$ and positive smoothing cuts a circle.
Fig.~\ref{fig:11_2} illustrates the other case: the sign of the crossing is equal to $-1$ and the smoothing cutting a circle is negative.
However, the computations proving invariance of the polynomial $J$ under the move $R_1$ is true for all possible situations
because, by our definitions, the type of smoothing and the sign of the crossing alter together.
Let us, as an example, briefly consider the situation shown in Fig.~\ref{fig:11_1}.

Denote by $D$ and $D'$ the diagram before and after adding the loop, respectively.
We have
\[
\begin{split}
J(D';u,v) =J(D;u,v)(u +u^{-1}(-u^2-u^{-2})) =\\
=-J(D;u,v) u^{-3} =(-u)^{-3} J(D;u,v).
\end{split}
\]
Hence, $J(D;u,v) =(-u)^3 J(D';u,v)$.
The latter equality implies that the move preserves the polynomial
because, in this situation, $w_1(D') =w_1(D) -1$.
Note that the above computation word-for-word coincides with the one in the classical case
when the crossing in question has the sign equal $-1$.

\begin{figure}[h!]
\centering 
\begin{subfigure}[b]{0.45\textwidth}
\centering
\psfig{figure=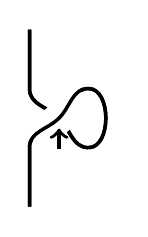, width=0.38\textwidth}
\caption{$\sign =1$}
\label{fig:11_1}
\end{subfigure}
\begin{subfigure}[b]{0.45\textwidth}
\centering
\psfig{figure=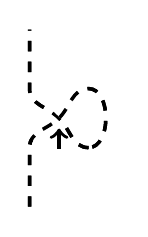, width=0.38\textwidth}
\caption{$\sign =-1$}
\label{fig:11_2}
\end{subfigure}
\caption{
The smoothing cutting a circle is positive in (a) and the one is negative in (b)
}
\label{fig:11}
\end{figure}

{\bf The move ${\mathbf{R_2}}$}.
Totally we have $16$ variants of the corresponding tangles
depending on the following:
\begin{itemize}
\item modifying strands are oriented in the same direction   or not,
\item what of them goes over,
\item  what components of $2$-cabling pass through input crossings,
\item what of these components goes over (if the two are distinct).
\end{itemize}
Both of the crossings which the move adds/removes are of the same type (either~$1$ or~$2$).
In all cases, one of two smoothings those cut the bigon and give an additional circle
is positive, while the other one is negative.
Using the observation, we can establish the invariance of $J(u,v)$ under the move $R_2$
by the same computations as in the classical situation.
Again, the computation coincides with the one in some of $4$ situations possible in classical case.
As an example, we consider a variant of the move depicted in Fig.~\ref{fig:12},
which illustrates the specificity of non-orientable surface.
Both of the crossings in Fig.~\ref{fig:12} are of the type~$2$,
$\sign(x) =i, \sign(y) =-i$. 
To obtain an additional circle, we need to perform the positive smoothing in $x$ and the negative one in $y$.
Two positive and two negative smoothings gives isotopic fragments.
The negative smoothing in $x$ and the positive smoothing in $y$ give a fragment isotopic to the fragment before the move.
Therefore, we have
\[
J(D';u,v) =u^{-1}u J(D;u,v) +J(D'';u,v)(u^2+u^{-2}-u^2-u^{-2}) =J(D;u,v),
\]
where $D$ and $D'$ --- diagrams before and after the move, respectively,
by $D''$ we denote the diagram that appears as a result of performing positive smoothings both in $x$ and in $y$ in the diagram $D'$.

\begin{figure}[h!]
\centerline{\psfig{figure=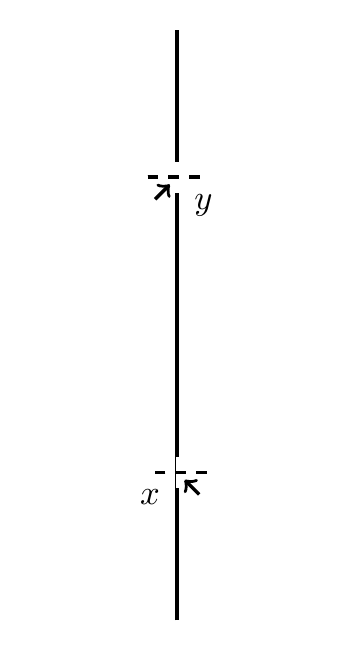}} 
\caption{An example of the move $R_2$}
\label{fig:12}
\end{figure}

{\bf The move ${\mathbf{R_3}}$}.
To perform the standard computation, which proves the invariance of the classical bracket polynomial under the move $R_3$,
it is sufficient to have the following two facts:
\begin{enumerate}
\item The modifying tangle both before and after the transformation admits a unique combination of smoothings
giving an additional circle,
and two of these three smoothings are positive while the third one is negative.
\item by construction an additional circle gives the additional factor $-t^2-t^{-2}$.
\end{enumerate}
Check that these facts hold in our situation.
Indeed, modifying tangle contains two crossings, in which the same strand of the diagram $D$ goes over,
and Lemma~\ref{lemma:PropertiesOfSmoothing}(3)
implies that to obtain an additional circle we need to perform different smoothings in these crossings
(positive in one of them and negative in the other).
In the case under consideration (unlike the classical one), the third smoothing in the combination giving the circle
can be both positive and negative.
If it is positive, the computation establishing the invariance of the polynomial $J$
word-for-word coincides with the one in the classical case.
The situation when the third smoothing is negative can be transformed into the previous one by the relabeling of $D^2$.
The transformation changes the polynomial $J$
(see Remark~\ref{remark:Relabeling}),
however, the invariance of the transformed polynomial under the move $R_3$ implies the invariance of the initial polynomial.

Finally, note that,
firstly, moves $R_2$ and $R_3$ does not change writhe numbers $w_1$ and $w_2$,
secondly, since there is no canonical labeling of $D^2$
we need to use the absolute value of the writhe number $w_2$ and to consider the polynomial up to the substitution $u =u^{-1}$.
This completes the proof of Theorem~\ref{theorem:J}.
\end{proof}

Let $K \subset \hat{\Sigma}$ be a non-oriented pseudo-classical knot.
In this case, we define $J(K;u,v)$ as $J(K;u,v) =J(k';u,v)$
where $K'$ is $K$ with an arbitrary orientation.
Then, by Theorem~\ref{theorem:J}
and Remark~\ref{remark:Relabeling},
we have the following:

\begin{corollary} \label{corollary:NonOriented}
The polynomial $J(K;u,v)$ is an isotopy invariant of a non-oriented pseudo-classical knot $K \subset \hat{\Sigma}$
up to the substitution $U =u^{-1}$.
\end{corollary}

\begin{remark} \label{remark:CrossingChange}
By Lemma~\ref{lemma:PropertiesOfSmoothing}(2),
the crossing change operation in a crossing of the type~$2$
does not affect the type of smoothing in the crossing
and, by Lemma~\ref{lemma:KnotSignProperties}(2),
it preserves the writhe number $w_2(D)$.
Hence, the crossing change operation preserves the polynomial $J$.
Therefore, the polynomial seems to be not very informative
if diagrams under consideration has many crossings of the type~$2$.
The example in Section~\ref{sec:J_ast is stronger} below
illustrates the fact.
\end{remark}

\subsection{A generalization of the polynomial $J$}
\label{sec:Generalization}

In \cite{KauffmanDye2005} Kauffman and Dye proposed a generalization of the Kauffman bracket polynomial
for knots and links in a thickened orientable surface.
The generalized polynomial takes value not in the ring of Laurent polynomial with integer coefficients
but in the free module (over the latter ring) generated by the set of homology (or homotopy) classes of circles on the surface.
The polynomial $J$ admits the same generalization.
To this end, we need to modify the monomial corresponding to each state of the diagram in question
(see~\eqref{eq:P}).
Namely, in the base version of the polynomial, each circle in the state adds the polynomial factor $-u^2-u^{-2}$.
In the generalized version, each homologically (or homotopically) trivial circle again adds to the monomial the factor $-u^2-u^{-2}$
while each non-trivial circle adds a factor that is the generator of the module corresponding to the homology (or homotopy) class of the circle.

\section{The polynomial $J(u,v)$ vs. Kauffman bracket polynomial of double cover}
\label{sec:vs}

Consider a non-orientable surface $\Sigma$
and a diagram $D \subset \Sigma$ of a knot $K \subset \hat{\Sigma}$.
Let an orientable  surface $\Sigma_*$ be the double cover of $\Sigma$
and $D_* \subset \Sigma_*$ denotes the diagram representing a knot (or a link) $K_* \subset (\Sigma_* \times [0,1])$
those are double covers of $D$ and $K$, respectively
(here speaking of  double covers of a diagram, a knot and a surface we mean mappings coming from the same double cover of the corresponding thickened surfaces).
It is well-known that there is a generalization of the Kauffman bracket polynomial
 for knots and links in a thickened orientable surface.
We denote the generalization by $J_*$.
Clearly, $J_*(K_*)$ is an invariant of $K$.
In the case when $K$ is pseudo-classical, we have defined the above invariant $J$,
and it is natural to ask
whether one of these two invariants ($J$ and $J_*$) is a consequence of the other in the particular case of  pseudo-classical knots?
In this section, we show that the answer is negative.
More precisely, below we consider two examples:
in the first one (Section~\ref{sec:J_ast is stronger}) $J_*$ distinguishes two given knots while $J$ does not,
and in the second one (Section~\ref{sec:J is stronger}), on the contrary, $J$ does while $J_*$ does not.

In both examples $\Sigma$ is the Klein bottle and $\Sigma_*$ is the torus.
In our figures the Klein bottle is represented by a rectangle
of which vertical sides are assumed to  be identified  via parallel translation
while its horizontal sides are assumed to be identified with a twist, i.e.,  via the composition of parallel translation
and the reflection over a vertical line.
The torus is represented by a rectangle, of which opposite sides (both pairs) are assumed to be identified via parallel translation.

In all cases, $D_*$ is obtained from $D$ via the following procedure:
we take two copies of a rectangle containing the diagram $D$,
reflect the second copy over a vertical line
and identify top side of the first copy  with bottom side of the second one.

Before considering examples, we need to make a remark on  the construction of the Kauffman bracket polynomial.
It is well-known that, in the classical case, one can define the Kauffman bracket polynomial for non-oriented  knots,
at the same time, in the case of links orientations of components matter.
The reason is that 
the factor that makes the polynomial invariant under the first Reidemeister move
depends on the writhe number of the diagram under consideration,
while the sign of an intersection point of distinct components
has sense for oriented links only.
To make the Kauffman bracket polynomial into an invariant for non-oriented links,
it is sufficient to ignore intersections of distinct components
when we compute the factor.
however, usually, the signs of these crossings are taken into account in the construction of the polynomial
and below we follow the tradition.
At the same time, in all cases below diagrams which are double cover
by construction (the union of a diagram with its mirror reflection) have writhe number equals zero
hence both approaches to the definition of the Kauffman bracket polynomial
(including and excluding the signs of intersection points of distinct components)
give polynomials having coinciding values on the diagrams in question.

\subsection{An example when $J_*$ is stronger than $J$}
\label{sec:J_ast is stronger}

Consider two diagrams on the Klein bottle:
$D_1$ (Fig.~\ref{fig:5_1}) and $D_2$ (Fig.~\ref{fig:5_2}).
Each of these diagrams has two intersections with the horizontal side of the rectangle representing the Klein bottle
(recall that horizontal sides are assumed to be identified with a twist)
hence knots represented by these diagrams are pseudo-classical,
and the polynomial $J$ has sense for both  of them.

Note that $D_1$ can be obtained from $D_2$
as a result of two transformations:
the crossing change operation in the crossing $u$
and then the second Reidemeister move removing $u$ and $v$.

Components of $2$-cabling of a diagram on the Klein bottle permute
when the diagram passes through the horizontal sides of the depicted rectangle,
hence the diagram $D_2$ has $3$ crossings (including $u$ and $v$) of the type~$2$.
As we mentioned above in Remark~\ref{remark:CrossingChange}
the crossing change operation in a crossing of the type~$2$
preserves the polynomial $J$.
The second Reidemeister move preserves $J$ also.
Therefore, both transformations which make $D_2$ into $D_1$ preserve the polynomial $J$,
thus $J(D_1) =J(D_2)$.

\begin{figure}[h!]
\centering 
\begin{subfigure}[b]{0.45\textwidth}
\centering
\psfig{figure=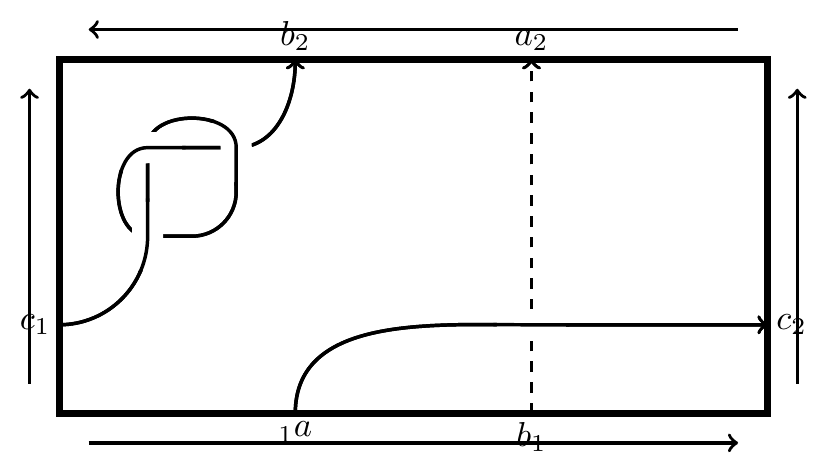, width=1\textwidth}
\caption{$D_1$}
\label{fig:5_1}
\end{subfigure}
\begin{subfigure}[b]{0.45\textwidth}
\centering
\psfig{figure=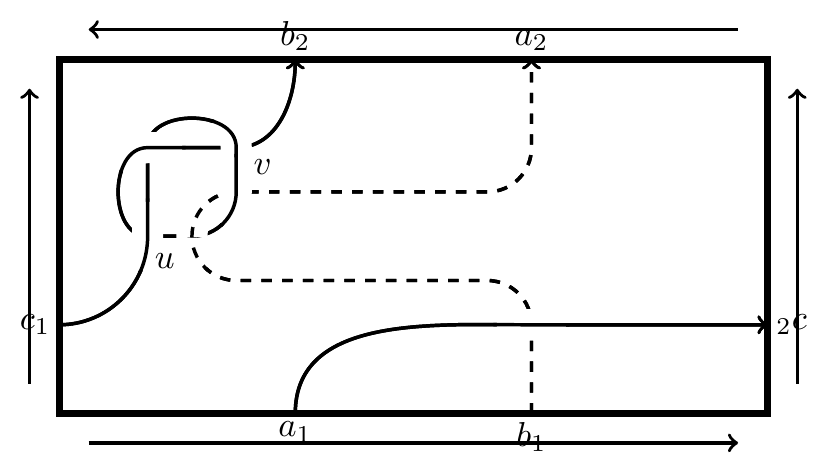, width=1\textwidth}
\caption{$D_2$}
\label{fig:5_2}
\end{subfigure}
\caption{Diagrams on the Klein bottle}
\label{fig:5}
\end{figure}

Diagrams $D_{1*}$ (Fig.~\ref{fig:6_1}) and $D_{2*}$ (Fig.~\ref{fig:6_2}) on the torus
are double covers of $D_1$ and $D_2$, respectively.
$D_{1*}$ has $8$ crossings, $D_{2*}$ has $12$ crossings,
 so the computation of $J_*(D_{1*})$ and, especially, $J_*(D_{2*})$ is too cumbersome
to be made by hand.
The values below are computed by ``3--manifold recognizer'' \cite{Recognizer} ---
a computer program for studying $3$--manifolds and knots.
\[
\begin{split}
J_*(D_{1*}) &=t^{-14}-2t^{-2}-2t^2+t^{14}, \\
J_*(D_{2*}) &=t^{-18}-t^{-14}+t^{-10}-t^{-6}-t^{-2}-t^2-t^6+t^{10}-t^{14}+t^{18}.
\end{split}
\]
(In Appendix below,
we give source data
needed for computing the above values in the format acceptable by the program ``3--manifold recognizer''.)
We see that these two polynomials are distinct.

\begin{figure}[h!]
\centering 
\begin{subfigure}[b]{0.45\textwidth}
\centering
\psfig{figure=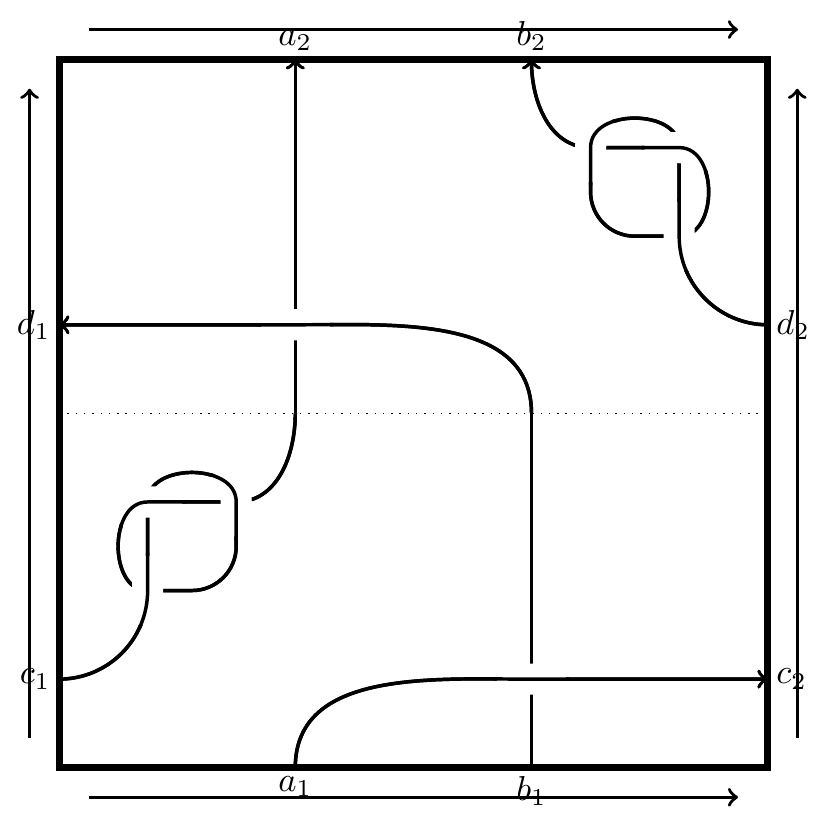, width=1\textwidth}
\caption{$D_{1*}$}
\label{fig:6_1}
\end{subfigure}
\begin{subfigure}[b]{0.45\textwidth}
\centering
\psfig{figure=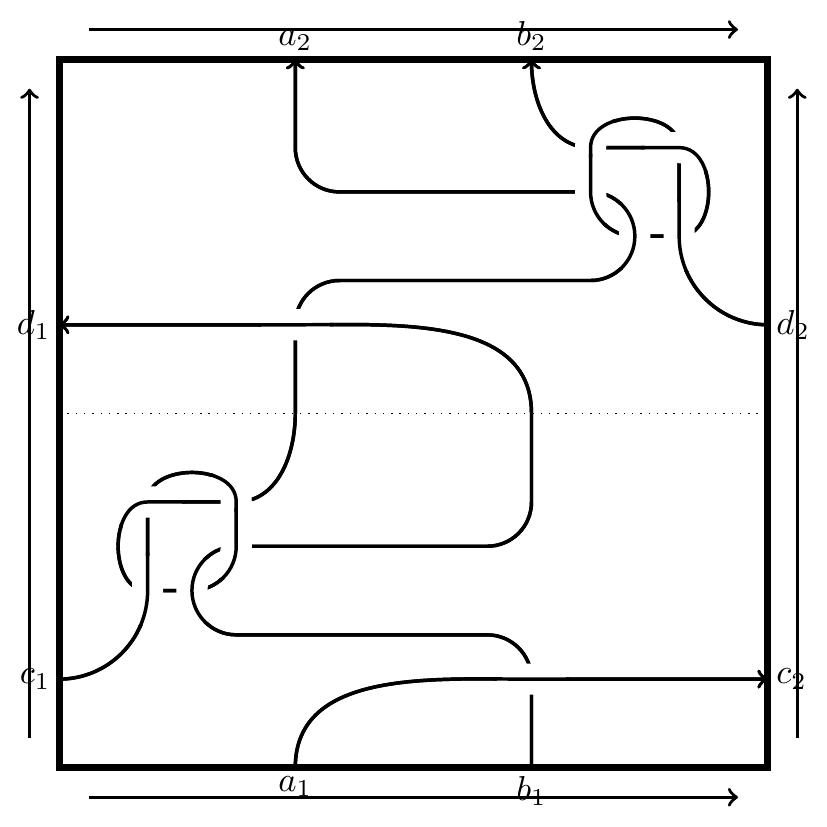, width=1\textwidth}
\caption{$D_{2*}$}
\label{fig:6_2}
\end{subfigure}
\caption{Diagrams on the torus which are the double cover of $D_1$ and $D_2$, respectively}
\label{fig:6}
\end{figure}

\begin{figure}[h!]
\centering 
\begin{subfigure}[b]{0.45\textwidth}
\centering
\psfig{figure=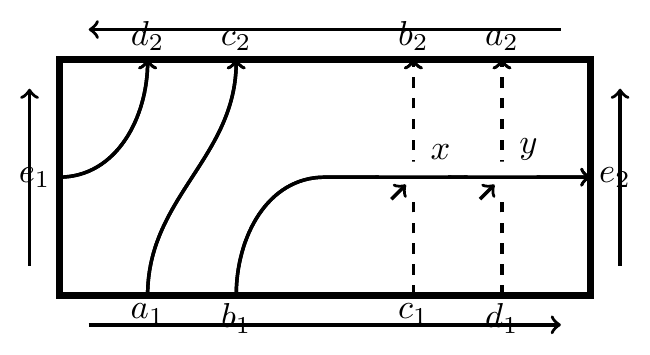, width=1\textwidth}
\caption{$D_3$}
\label{fig:7_1}
\end{subfigure}
\begin{subfigure}[b]{0.45\textwidth}
\centering
\psfig{figure=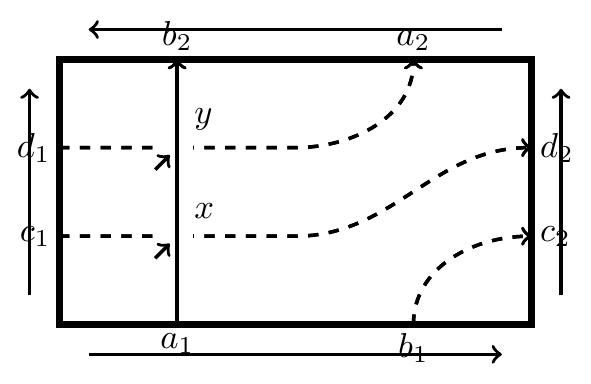, width=1\textwidth}
\caption{$D_4$}
\label{fig:7_2}
\end{subfigure}
\caption{Diagrams on the Klein bottle so that $J(D_3) \neq J(D_4)$.}
\label{fig:7}
\end{figure}

\begin{figure}[h!]
\centering 
\begin{subfigure}[b]{0.45\textwidth}
\centering
\psfig{figure=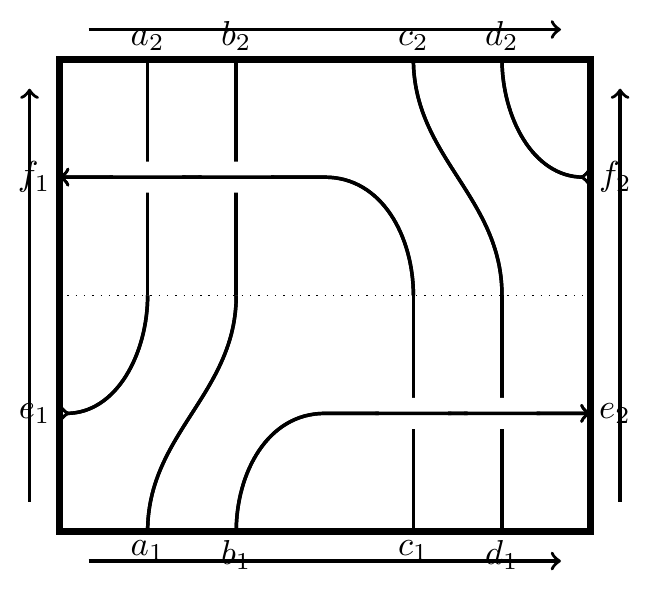, width=1\textwidth}
\caption{$D_{3*}$}
\label{fig:8_1}
\end{subfigure}
\begin{subfigure}[b]{0.45\textwidth}
\centering
\psfig{figure=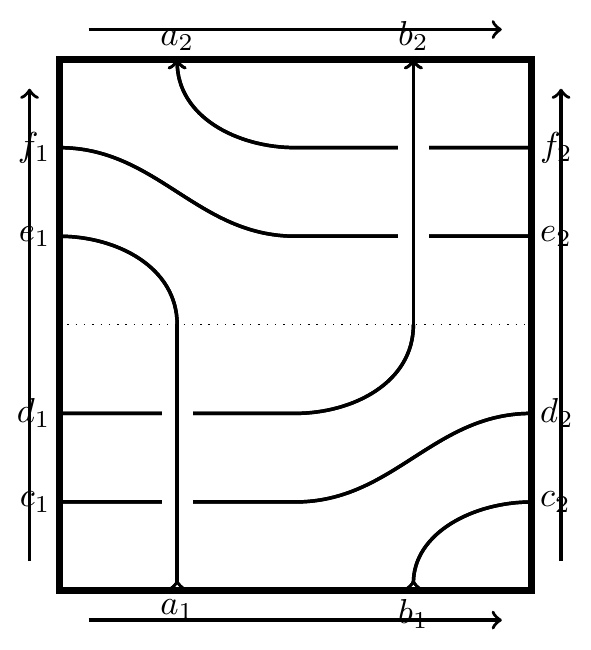, width=1\textwidth}
\caption{$D_{4*}$}
\label{fig:8_2}
\end{subfigure}
\caption{Diagrams on the torus, which are the double cover of $D_3$ and $D_4$, respectively.
$D_{4*}$ obtains (up to orientation of components) from $D_{3*}$ by rotation by $90^{\circ}$ in the clockwise direction.}
\label{fig:8}
\end{figure}

\subsection{An example when $J$ is stronger than $J_*$}
\label{sec:J is stronger}

The diagrams $D_3$ (Fig.~\ref{fig:7_1}) and $D_4$ (Fig.~\ref{fig:7_2}) on the Klein bottle
are so that $J(D_3) \not=J(D_4)$
(the fact will be established below).
At the same time, the diagrams $D_{4*}$ (Fig.~\ref{fig:8_1}) and $D_{4*}$ (Fig.~\ref{fig:8_2}) on the torus,
which are double covers of $D_3$ and $D_4$, respectively,
represent $2$-component links in the thickened torus,
which have coinciding values of $J_*$.
To establish the equality $J_*(D_{3*}) =J_*(D_{4*})$, 
it is sufficient to observe that
$D_{4*}$ obtains from $D_{3*}$ as a result of the  rotation by $90^{\circ}$ in the clockwise direction
and the reversing the orientation of that component
which in Fig.~\ref{fig:8_1} passes through points $a_1=a_2,b_1=b_2,c_1=c_2$.
The fact implies the required equality
because, as we mentioned in the beginning of Section~\ref{sec:vs},
in the case under consideration, the value of the polynomial $J_*$ does not depend on the orientation of link components.
One can prove that links in the thickened torus represented by $D_{3*}$ and $D_{4*}$ are inequivalent as oriented links
but it is not necessary in our context.

Now, we compute $J(D_3)$ and $J(D_4)$.

\subsubsection{Computation of $J(D_3)$} The diagram $D_3$ (Fig.~\ref{fig:7_1}) has $2$ crossings denoted by $x$ and $y$ hence
to compute $J(D_3)$ we need to consider $4$ states $S_1,\ldots,S_4$ of the diagram 
which map the pair of its crossings $(x,y)$, respectively, to  the pairs
$(1,1)$, $(1,-1)$, $(-1,1)$, $(-1, -1)$.
Below, circles those appear as a result of smoothings  $n(S_j)$ (for shortness the parameter  $S_j$ will be omitted)
are presented as the union of arcs having endpoints in the sides of the depicted rectangle.
For the definition of the polynomial $P(S;u)$ corresponding to each state of a diagram see~\eqref{eq:P}.

$S_1$: $(1, 1)$
$n(S_1)=3$, 
\\$\gamma_1=[a_1, c_2] \cup [c_1,a_2]$,
$\gamma_2 =[b_1,b_2]$,
$\gamma_3 =[d_1, e_2] \cup [e_1, d_2]$, 
\\$P(S_1;u)=
u^2(-u^2-u^{-2})^2
=u^6+2u^2+u^{-2}$;

$S_2$:  $(1, -1)$
$n(S_2)=2$, 
\\$\gamma_1=[a_1,c_2] \cup [c_1,d_1] \cup [d_2,e_1] \cup [e_2,a_2]$,
$\gamma_2=[b_1,b_2]$,
\\$P(S_2;u)=
u \, u^{-1}(-u^2-u^{-2}) 
=-u^2 -u^{-2}$;

$S_3$:  $(-1, 1)$
$n(S_3)=2$, 
\\$\gamma_1=[a_1,c_2] \cup [c_1,b_1] \cup [b_2,a_2]$,
$\gamma_2=[d_1,e_2] \cup [e_1,d_2]$,
\\$P(S_3;u)=
u^{-1} u(-u^2-u^{-2}) 
=-u^2 -u^{-2}$;

$S_4$:  $(-1, -1)$
$n(S_4)=1$,
\\$\gamma_1=[a_1,c_2] \cup [c_1,b_1] \cup [b_2,d_1] \cup [d_2,e_1] \cup [e_2,a_2]$,
\\$P(S_4;u)=
u^{-2}$.

To compute the resulting value of $J(D_3)$ we need to know values of writhe numbers $w_1$ (see~\eqref{eq:w_1}) and $w_2$ (see~\eqref{eq:w_2}).
The diagram $D_3$ has $2$ crossings, both  are of the type~$2$,
hence $w_1(D_3)=0$.
Input crossings $\Input(x)$ and $\Input(y)$ are self-intersections of $\Right{D_3}$
thus $\sign(x) =\sign(y) =i$ and $w_2(D_3)=2$.

Therefore (see~\eqref{eq:J}), 

\[
\begin{split}
J(D_3;u,v) 
=u^{0}v^2 \sum P(S_j,u) =\\ 
=v^2(
u^6+2u^2+u^{-2}
-u^2 -u^{-2}
-u^2 -u^{-2}
+u^{-2})  =\\
=u^6v^2.
\end{split}
\]

\subsubsection{Computation of $J(D_4)$}
The diagram $D_4$ (Fig.~\ref{fig:7_2}) has $2$ crossings denoted by $x$ and $y$ hence
to compute $J(D_4)$ we need to consider $4$ states of the diagram.

$S_1$: $(1, 1)$
$n(S_1)=1$,
\\$\gamma_1=[a_1,f_2] \cup [f_1,b_2] \cup [b_1,e_2] \cup [e_1,a_2]$,
\\$P(S_1;u)
=u^2$;

$S_2$:  $(1, -1)$
$n(S_2)=1$,
\\$\gamma_1=[a_1,f_2] \cup [f_1,e_1] \cup [e_2,b_1] \cup [b_2,a_2]$,
\\$P(S_2;u)
=u\, u^{-1}
=1$;

$S_3$:  $(-1, 1)$
$n(S_3)=1$, 
\\$\gamma_1=[a_1,e_1] \cup [e_2,b_1] \cup [b_2,f_1] \cup [f_2,a_2]$,
\\$P(S_3;u)
=u^{-1}u
=1$;

$S_4$:  $(-1, -1)$
$n(S_4)=2$,
\\$\gamma_1=[a_1,e_1] \cup [e_2,b_1] \cup [b_2,a_2]$,
$\gamma_2=[f_1,f_2]$,
\\$P(S_4;u)
=u^{-2}(-u^2-u^{-2})
=-1-u^{-4}$.

The writhe numbers have the same values as above,
hence
\[
\begin{split}
J(D_4;u,v)
=v^2(
u^2
+1
+1
-1-u^{-4}
) =\\
=v^2(u^2+1-u^{-4}).
\end{split}
\]

Therefore, $J(D_3) \not=J(D_4)$
and the substitution $u =u^{-1}$ does not transform one of them into another.
Thus, by Theorem~\ref{theorem:J},
 knots in (the Klein bottle) $\times [0,1]$ represented by these diagrams are inequivalent.

\section*{Appendix: Source data for a computation via the computer program ``3--manifold recognizer''}
\label{sec:Appendix}

To compute the Kauffman bracket polynomial of a knot (or a link) in the thickened torus
via the computer program ``3--manifold recognizer''
\cite{Recognizer}
it is necessary to describe the  knot (or link) in question using an extended Gauss code.
Namely, we begin with numbering of the crossings of the diagram by natural numbers from $1$ in increasing order.
The crossings of the diagram with sides of the rectangle in which the diagram lies
are denoted by a letter and a number which follows the letter without space between them.
Letters $l,t,r,b$ mean that the corresponding crossing lies on the left, top, right, bottom side of the rectangle, respectively.
The number is the number of the crossing in corresponding side.
For all $4$ sides, these numbers increase in the clockwise direction.
For example, l1 denotes the lowermost intersection with the left side
while r1 denotes the topmost intersection with the right one,
t1 denotes the leftmost intersection with the top side
while b1 denotes the rightmost intersection with the bottom one.
The meaning of all other lines in the following description of the links
seems to be clear.

{\bf The description of the diagram $\mathbf{D_1*}$}.
(Fig~\ref{fig:6_1})

\begin{verbatim}
link  T^2
crossings 8
signs 1 1 1 1 -1 -1 -1 -1
code 1 -2 3 -1 2 -3 -8 t1 b2 4 r2 l1
code -4 8 l2 r1 5 -6 7 -5 6 -7 t2 b1
end
 \end{verbatim}

{\bf The description of the diagram $\mathbf{D_2*}$}.
(Fig~\ref{fig:6_2})

\begin{verbatim}
link  T^2
crossings 12
signs 1 1 1 1 -1 -1 -1 -1 1 1 -1 -1
code 1 -2 3 9 -10 -1 2 -3 -8 12 -11 t1 b2 4 r2 l1
code -4 10 -9 8 l2 r1 5 -6 7 11 -12 -5 6 -7 t2 b1
end
\end{verbatim}


\end{document}